\RequirePackage[l2tabu, orthodox]{nag}

\documentclass[12pt]{amsart}

\usepackage{fullpage,url,colonequals}
\usepackage{rotate,amsmath,amscd,latexsym,amsfonts,amsthm}
\usepackage{amssymb,enumerate,mathrsfs,verbatim}
\usepackage[all,cmtip]{xy}

\usepackage{color}

\newcommand{\defi}[1]{\textsf{#1}} 

\newtheorem{theorem}{Theorem}[section]
\newtheorem{lemma}[theorem]{Lemma}

\newtheorem{proposition}[theorem]{Proposition}

\theoremstyle{definition}

\theoremstyle{remark}
\newtheorem{remark}[theorem]{Remark}

\newcommand{\Adeles}{\mathbb{A}}
\newcommand{\Ltilde}{\widetilde{L}}
\newcommand{\Xtilde}{\widetilde{X}}
\newcommand{\Ctilde}{\widetilde{C}}

\newcommand{\fp}{{\mathfrak{p}}}

\newcommand{\fP}{\mathfrak{P}}
\newcommand{\fc}{\mathfrak{c}}
\newcommand{\sO}{\mathcal{O}}
\newcommand{\Otilde}{\widetilde{\sO}}

\newcommand{\Z}{\mathbb{Z}}
\newcommand{\F}{\mathbb{F}}

\newcommand{\R}{\mathbb{R}}
\newcommand{\RR}{\mathbb{R}}
\newcommand{\C}{\mathbb{C}}
\newcommand{\CC}{\mathbb{C}}

\newcommand{\Q}{\mathbb{Q}}
\newcommand{\QQ}{\mathbb{Q}}

\newcommand{\calO}{\mathcal{O}}
\newcommand{\scriptF}{\mathscr{F}}
\newcommand{\OO}{\mathscr{O}}

\newcommand{\Directsum}{\bigoplus}

\newcommand{\isom}{\simeq}

\newcommand{\intersect}{\cap}

\newcommand{\injects}{\hookrightarrow}
\newcommand{\surjects}{\twoheadrightarrow}

\newcommand{\To}{\longrightarrow}

\DeclareMathOperator{\Char}{char}

\DeclareMathOperator{\Disc}{Disc}
\DeclareMathOperator{\Div}{Div}
\DeclareMathOperator{\Divhat}{\widehat{Div}}

\DeclareMathOperator{\Frac}{Frac}

\DeclareMathOperator{\im}{im}

\DeclareMathOperator{\nonarch}{nonarch}

\DeclareMathOperator{\Pic}{Pic}
\DeclareMathOperator{\Pichat}{\widehat{Pic}}

\DeclareMathOperator{\Spec}{Spec}

\DeclareMathOperator{\Tr}{Tr}
\DeclareMathOperator{\vol}{vol}

\usepackage{microtype}
\usepackage[
	pdfauthor={Bruce Jordan, Bjorn Poonen},
        pdftitle={The analytic class number formula for orders in products of number fields},
]{hyperref}
\usepackage[alphabetic,lite,nobysame]{amsrefs} 

\begin{document}

\title{The analytic class number formula for $1$-dimensional affine schemes}
\author{Bruce~W.~Jordan}
\address{Department of Mathematics, Baruch College, The City University
of New York, One Bernard Baruch Way, New York, NY 10010-5526, USA}
\email{bruce.jordan@baruch.cuny.edu}

\author{Bjorn~Poonen}
\address{Department of Mathematics, Massachusetts Institute of Technology, Cambridge, MA 02139-4307, USA}
\email{poonen@math.mit.edu}
\urladdr{\url{http://math.mit.edu/~poonen/}}

\thanks{The second author was supported in part by National Science Foundation grant DMS-1069236 and DMS-1601946 and grants from the Simons Foundation (\#340694, \#402472, and \#550033).}

\date{December 30, 2019}

\begin{abstract}
We derive an analytic class number formula
valid for an order in a product of $S$-integers in global fields,
or equivalently for reduced finite-type affine schemes of pure dimension~$1$ 
over $\Z$.
\end{abstract}

\maketitle

\section{Introduction}
\label{S:introduction}

Let $K$ be a finite extension of $\Q$.
Let $\calO_K$ be its ring of integers.
Let $\zeta_K(s)$ be the Dedekind zeta function of $K$,
which is the zeta function of $\Spec \calO_K$.
Dedekind \cite{Dirichlet1894}*{Supplement~XI, \S184,~IV},
generalizing work of Dirichlet,
proved the \defi{analytic class number formula},
which expresses the residue of $\zeta_K(s)$ at $s=1$
in terms of arithmetic invariants 
(see also Hilbert's \emph{Zahlbericht} \cite{Hilbert1897}*{Theorem~56}).
More precisely, he proved that
\begin{equation}
\label{E:classical analytic class number formula}
	\lim_{s \to 1} (s-1) \zeta_K(s) 
	= \frac{ 2^{r_1} \, (2\pi)^{r_2} \, hR}
		{w \, \lvert \Disc K \rvert^{1/2}},
\end{equation}
where $r_1$ is the number of real places, 
$r_2$ is the number of complex places,
$h$ is the class number,
$R$ is the unit regulator,
$w$ is the number of roots of unity,
and $\Disc K$ is the discriminant.
F.~K.~Schmidt~\cite{Schmidt1931}*{Satz~21} 
proved an analogue for a global function field.

This could be generalized in several ways:
\begin{itemize}
\item Replace $\calO_K$ by a non-maximal order.
\item Replace $\calO_K$ by a ring of $S$-integers 
for some finite set $S$ of places of $K$.
\item Allow $\Z$-algebras of Krull dimension~$1$ that are not
necessarily integral domains.
\end{itemize}
We will generalize simultaneously in all of these directions,
by proving a version of \eqref{E:classical analytic class number formula}
for an order $\calO$ in a product of $S$-integers in global fields,
or equivalently for 
a reduced affine finite-type $\Z$-scheme of pure dimension~$1$
(for the equivalence, see Proposition~\ref{P:characterization of rings}).
Our main result, 
expressing the leading term of the arithmetic zeta function 
of \cite{Serre1965}*{p.~83}
in terms of quantities to be defined in Section~\ref{S:invariants}, 
is this:

\begin{theorem}[Generalized analytic class number formula]
\label{T:main}
Let $X$ be a reduced affine finite-type $\Z$-scheme of pure dimension~$1$,
say $X = \Spec \sO$.
Let $m$ be the number of irreducible components of $X$.
Then
\begin{equation}
\label{E:sweet potato}
	\lim_{s\rightarrow 1} (s-1)^m \zeta_X(s)
	= \frac{ 2^{r_1} \, (2\pi)^{r_2} \, h(\sO) \, R(\sO) 
			\prod_{v \in S_{\nonarch}} \left( \left( 1-q_v^{-1} \right)/ \log q_v \right)}
		{w(\sO) \, \lvert \Disc \sO \rvert^{1/2} }.
\end{equation}
\end{theorem}

\begin{remark}
Theorem~\ref{T:main} seems to be new even 
for orders in rings of integers of number fields
and for rings of $S$-integers in number fields.
\end{remark}

\begin{remark}
Replacing a finite-type $\Z$-scheme $X$ 
by its associated reduced subscheme $X_{\textup{red}}$ 
does not change $\zeta_X(s)$.
On the other hand,
if the scheme $X = \Spec \sO$ in Theorem~\ref{T:main} 
is not assumed to be reduced, then $\sO^\times$ need not be
finitely generated, so defining the regulator $R(\sO)$ is problematic:
for example, if $\sO \colonequals \F_2[t,\epsilon]/(\epsilon^2)$,
then $\sO^\times$ is isomorphic to the additive group of $\F_2[t]$
via the homomorphism $1+f\epsilon \mapsto f$.
\end{remark}

\begin{remark}
Various authors defined other zeta functions attached to a singular curve
over a finite field and computed their leading terms at $s=0$ or $s=1$ 
\cites{Galkin1973,Green1989,Zuniga-Galindo1997a,Zuniga-Galindo1997b,Stoehr1998}, 
but these zeta functions are different from
the usual zeta function in~\cite{Serre1965} in general.
\end{remark}

\subsection{Outline of the article.}
Section~\ref{S:Tate} proves Theorem~\ref{T:main} for a ring of $S$-integers.
Because we want a formula involving the 
$S$-class group and $S$-regulator instead of 
the usual class group and regulator,
it is not convenient to deduce this from the classical formula
for the ring of integers.
Instead we redo the calculations of Tate's thesis for $S$-integers.

Sections \ref{S:characterizing} and~\ref{S:describing}
characterize the rings $\sO$ such that 
$\Spec \sO$ is a reduced affine finite-type $\Z$-scheme of pure dimension~$1$.
In particular, the normalization $\Otilde$ of such a ring $\sO$
is a product of rings of $S$-integers.

Section~\ref{S:invariants} defines all the quantities that appear
in~\eqref{E:sweet potato}.

Section~\ref{S:nonmaximal} proves Theorem~\ref{T:main}.
The formula for $\Otilde$ follows from the case proved in Section~\ref{S:Tate},
so our strategy is to determine how each term in \eqref{E:sweet potato}
changes when $\sO$ is replaced by $\Otilde$.
In particular, we use the Leray spectral sequence to determine
how the unit group and Picard group change.

Finally Sections \ref{S:example 1} and~\ref{S:example 2} 
illustrate \eqref{E:sweet potato} 
in examples exhibiting the phenomena that can arise in our context: 
both number fields and function fields, 
$S$-integers instead of just integers,
multiple irreducible components, 
and non-maximal orders (and hence singular points of the scheme).

\section{Tate's thesis for \texorpdfstring{$S$}{S}-integers}
\label{S:Tate}

Let $K$ be a global field.
Let $\mu$ be the torsion subgroup of $K^\times$, and let $w = \#\mu$.
For each place $v$ of $K$, let $K_v$ be the completion of $K$ at $v$.
If $K_v \isom \R$, equip it with Lebesgue measure.
If $K_v \isom \C$, equip it with $2$ times Lebesgue measure.
For each nonarchimedean $v$, let $\sO_v$ be the valuation ring in $K_v$, 
let $q_v$ be the size of its residue field,
and equip $K_v$ with the Haar measure $dx$ for which $\vol(\sO_v)=1$; 
here we follow \cite{Weil1967}*{p.~95} instead of taking the self-dual
measure as in \cite{Tate1967-thesis}*{p.~310}.
We write $\vol(T)$ for the measure of a set $T$ with respect to a measure
that is implied by context.

If $a \in K_v$ for some $v$, let $|a|_v \in \R_{\ge 0}$ be the factor by which
multiplication-by-$a$ scales the Haar measure on $K_v$.
The measure we use on $K_v^\times$ is not the restriction of 
the measure $dx$ on $K_v$.
If $v$ is archimedean, equip $K_v^\times$ with the Haar measure $dx/|x|_v$.
If $v$ is nonarchimedean, equip $K_v^\times$ with the Haar measure 
for which $\vol(\sO_v^\times)=1$.

Let $K_{v,1}^\times \colonequals \{x \in K_v^\times : |x|_v=1\}$.
If $v$ is nonarchimedean, $K_{v,1}^\times = \sO_v^\times$,
which has volume $1$ for the Haar measure on $K_v^\times$.
If $v$ is archimedean, then equip $K_{v,1}^\times$ with the Haar measure
compatible with the Haar measure on $K_v^\times$ and Lebesgue measure on $\R$
in the exact sequence
\[
	1 \To K_{v,1}^\times \To K_v^\times \stackrel{\log |\;|_v}\To \R \To 0
\]
(for the notion of compatibility, see, e.g., \cite{Deitmar-Echterhoff2014}*{Theorem~1.53}; the notion can be extended to exact sequences of arbitrary finite length by breaking them into short exact sequences).

Define the \defi{ad\`ele ring} $\Adeles$ 
as the restricted product ${\prod}'_v (K_v,\sO_v)$
with the product measure.
Equip the \defi{id\`ele group} 
$\Adeles^\times = {\prod}'_v (K_v^\times,\sO_v^\times)$
with the product of the multiplicative Haar measures.
The field $K$ embeds diagonally in $\Adeles$,
and $K^\times$ embeds diagonally in $\Adeles^\times$.
Equip the discrete groups $K$ and $K^\times$ and their subgroups 
with the counting measure,
in order to equip $\Adeles/K$ and $\Adeles^\times/K^\times$ with measures.

Let $S$ be a finite nonempty set of places of $K$ 
containing all the archimedean places.
Let $S_{\nonarch}$ be the set of nonarchimedean places in $S$.
Define the \defi{ring of $S$-integers} by 
\[
	\sO \colonequals \{x \in K : v(x) \ge 0 \textup{ for all $v \notin S$} \}.
\]

\begin{lemma}
\label{L:adele quotient}
We have a measure-compatible isomorphism
\[
	\frac{\prod_{v \in S} K_v \times \prod_{v \notin S} \sO_v}{\sO} 
	\stackrel{\sim}\To \frac{\Adeles}{K}
\]
and a measure-compatible short exact sequence
\[
	0 
	\To \prod_{v \notin S} \sO_v 
	\To \frac{\prod_{v \in S} K_v \times \prod_{v \notin S} \sO_v}{\sO} 
	\To \frac{\prod_{v \in S} K_v}{\sO}
	\To 0.
\]
\end{lemma}

\begin{proof}
Since $S$ is nonempty,
strong approximation \cite{Cassels1967}*{\S15} 
implies that any $x \in \Adeles$ can be written as $y+\epsilon$
with $y \in K$ and $\epsilon = (\epsilon_v) \in \Adeles$
such that $\epsilon_v \in \sO_v$ for all $v \notin S$.
In other words, the homomorphism
\[
	\prod_{v \in S} K_v \times \prod_{v \notin S} \sO_v \To \frac{\Adeles}{K}
\]
is surjective.
Its kernel is $K \intersect \left( \prod_{v \in S} K_v \times \prod_{v \notin S} \sO_v \right) = \sO$, 
so we obtain the claimed isomorphism.
By definition of the measures, the upper three homomorphisms in the diagram
\[
\xymatrix{
	\prod_{v \in S} K_v \times \prod_{v \notin S} \sO_v \ar@{^{(}->}[r] \ar[d] & \Adeles \ar[d] \\
	\dfrac{\prod_{v \in S} K_v \times \prod_{v \notin S} \sO_v}{\sO} \ar[r]^-{\sim} & \dfrac{\Adeles}{K}
}
\]
respect the measures, so the induced isomorphism at the bottom does too.

The exact sequence is obtained from the measure-compatible split exact sequence
\[
	0 
	\To \prod_{v \notin S} \sO_v 
	\To \prod_{v \in S} K_v \times \prod_{v \notin S} \sO_v
	\To \prod_{v \in S} K_v
	\To 0
\]
by dividing each of the last two terms by the image of $\sO$ 
with the counting measure.
\end{proof}

Define the \defi{$S$-Arakelov divisor group} 
(cf.\ the usual Arakelov divisor group in \cite{Neukirch1999}*{III.1.8})
by 
\[
	\Divhat \sO \colonequals \Directsum_{v \in S} \R \times \Directsum_{v \notin S} (\Z \log q_v),
\]
where each $\R$ has Lebesgue measure
and each $\Z \log q_v$ has the counting measure.
We have a homomorphism $\deg \colon \Divhat \sO \to \R$
that sums the components,
and its kernel is denoted $\Divhat^0 \sO$,
which acquires a measure compatible with the sequence 
$0 \to \Divhat^0 \sO \to \Divhat \sO \to \R \to 0$.
The homomorphism $\Adeles^\times \to \Divhat \sO$
sending $(x_v)$ to $(\log |x_v|_v)$
restricts to a homomorphism $K^\times \to \Divhat^0 \sO$
whose cokernel is called 
the \defi{$S$-Arakelov class group} $\Pichat^0 \sO$.
Let $\Adeles^\times_1$ be the kernel of the composition 
$\Adeles^\times \to \Divhat \sO \stackrel{\deg}\to \R$.
Equip $\R^S \colonequals \Directsum_{v \in S} \R$ and $\R$
with Lebesgue measure.
Equip the sum-zero hyperplane 
$\R^S_0 \colonequals \ker\left(\R^S \stackrel{\textup{sum}}\To \R\right)$
with the measure compatible with those
(equivalently, identify $\R^S_0$ with its projection
under the forget-one-coordinate map and use Lebesgue measure on the image).

\begin{lemma}
\label{L:Pic-hat}
We have a measure-preserving exact sequence
\[
	0
	\To \frac{\R^S_0} {\im \sO^\times}
	\To \Pichat^0 \sO
	\To \Pic \sO
	\To 1.
\]
\end{lemma}

\begin{proof}
Apply the snake lemma to
\[
\begin{gathered}[b]
\xymatrix{
1 \ar[r] & \sO^\times \ar[r] \ar[d]
	& K^\times \ar[r] \ar[d]
	& K^\times/\sO^\times \ar[r] \ar[d]
	& 1 \\
0 \ar[r] & \R^S_0 \ar[r] 
	&\Divhat^0 \sO \ar[r]
	& \Div \sO \ar[r] 
	& 0\lefteqn{.}
}\\[-\dp\strutbox]
\end{gathered}
\qedhere
\]
\end{proof}

If $\Char K>0$, let $q$ be the size of the constant field of $K$.
In all the formulas below, the term in square brackets 
involving $\log q$ should be present only if $\Char K > 0$.
Each copy of $\R$ has Lebesgue measure, 
which induces a measure on quotients such as $\R/(\Z \log q)$.

\begin{lemma}
\label{L:idele class group}
We have a measure-compatible exact sequence
\[
	1 
	\To \frac{\prod_v K_{v,1}^\times}{\mu}
	\To \frac{\Adeles^\times_1}{K^\times}
	\To \Pichat^0 \sO
	\To \Directsum_{v \in S_{\nonarch}} \frac{\R}{\Z \log q_v}
	\To \left[ \frac{\R}{\Z \log q} \right]
	\To 0.
\]
\end{lemma}

\begin{proof}
Since $S$ contains all archimedean places,
we have a commutative diagram 
\[
\xymatrix{
   0 \ar[r] 
      & \displaystyle \Directsum_{\textup{$v$ arch}} \R \times \!\!\!\! \Directsum_{\textup{$v$ nonarch}} \Z \log q_v \ar[r] \ar[d]^-{\textup{sum}}
      & \displaystyle \Directsum_{v \in S} \R \times \Directsum_{v \notin S} \Z \log q_v \ar[r] \ar[d]^-{\textup{sum}}
      & \displaystyle \Directsum_{v \in S_{\nonarch}} \frac{\R}{\Z \log q_v} \ar[r] \ar[d] 
      & 0 \\
   0 \ar[r] 
      & \R \ar@{=}[r] 
      & \R \ar[r]
      & 0
}
\]
with exact measure-compatible rows.
The second vertical homomorphism is $\deg \colon \Divhat \sO \to \R$,
so it is surjective with kernel $\Divhat^0 \sO$.
The upper left group is 
$\im\left(\Adeles^\times \to \Divhat \sO \right)$,
so the kernel of the first vertical homomorphism is 
$\im\left(\Adeles^\times_1 \to \Divhat^0 \sO \right)$.
If $K$ is a number field, then the first vertical homomorphism is surjective.
If $K$ is a function field, then its image is $\Z \log q$
since a smooth projective curve over $\F_q$ has closed points
of every sufficiently large degree.
Thus the snake lemma explains exactness at the last three nontrivial positions
in the exact sequence
\[
	1
	\To \prod_v K_{v,1}^\times 
	\To \Adeles^\times_1 
	\To \Divhat^0 \sO 
	\To \Directsum_{v \in S_{\nonarch}} \frac{\R}{\Z \log q_v}
	\stackrel{\textup{snake}}\To \left[ \frac{\R}{\Z \log q} \right]
	\To 0,
\]
and exactness at the first two positions follows since
$K^\times_{v,1}$ is the kernel of $\log |\;|_v \colon K_v^\times \to \R$.
The discrete subgroup $K^\times$
and compact subgroup $\prod_v K_{v,1}^\times$ of $\Adeles^\times_1$ 
intersect in $\mu$, 
so forming quotients yields the claimed exact sequence.
By construction, each of the exact sequences above is measure-compatible.
\end{proof}

If $K$ is a number field, 
let $(e_i)_{1 \le i \le n}$ be a basis for the ring of integers of $K$,
and define $\Disc K \colonequals \det (\Tr(e_i e_j))_{1 \le i,j \le n} \in \Z$.
If $K$ is a global function field of genus $g$ over $\F_q$,
define $\Disc K \colonequals q^{2g-2}$;
this is so that in both cases it is $\lvert \Disc K \rvert^{s/2}$ times
the completed zeta function below that is symmetric with respect to 
$s \mapsto 1-s$.
Define the \defi{$S$-class number} $h(\sO) \colonequals \# \Pic \sO$.
By the proof of the Dirichlet $S$-unit theorem, 
the image of $\sO^\times \to \R^S_0$ is a full lattice;
its covolume is called the \defi{$S$-regulator} $R(\sO)$.

\begin{lemma}
\label{L:volume calculations}
We have
\begin{align*}
\vol\left(\frac{\Adeles}{K}\right) 
	&= \lvert \Disc K \rvert^{1/2} \\
\vol\left(\frac{\prod_{v \in S} K_v}{\sO}\right) 
	&= \lvert \Disc K \rvert^{1/2} \\
\vol\left( K_{v,1}^\times \right) 
	&= 
	\begin{cases}
	2, & \textup{ if $K_v \isom \RR$;} \\
	2\pi, & \textup{ if $K_v \isom \CC$;} \\
	1, & \textup{ if $K_v$ is nonarchimedean.}
	\end{cases} \\
\vol\left( \Pichat^0 \sO \right) 
	&= h(\sO) \, R(\sO) \\
\vol\left( \frac{\Adeles^\times_1}{K^\times} \right) 
	&= \frac{ 2^{r_1} (2\pi)^{r_2} \, h(\sO) \, R(\sO) \, [\log q]}
		{w(\sO) \prod_{v \in S_{\nonarch}} \log q_v}.
\end{align*}
$($Again, the term in square brackets should be present only if $\Char K>0$.$)$
\end{lemma}

\begin{proof}
The first formula can be found in \cite{Weil1982}*{\S2.1.3}.
It implies the second, by Lemma~\ref{L:adele quotient}.
For $\vol(K_{v,1}^\times)$ for archimedean $v$,
see Tate's thesis \cite{Tate1967-thesis}*{p.~337}.
For nonarchimedean $v$, the group $K_{v,1}^\times=\sO_v^\times$ 
has volume~$1$ by definition of the measure.
Lemma~\ref{L:Pic-hat} computes $\vol\left( \Pichat^0 \sO \right)$.
Lemma~\ref{L:idele class group} yields the last formula.
\end{proof}

Define gamma factors $\Gamma_\R(s) \colonequals \pi^{-s/2} \Gamma(s/2)$
and $\Gamma_\C(s) \colonequals (2\pi)^{1-s} \Gamma(s)$,
and define the \defi{completed zeta function} by
\[
	\widehat{\zeta}_K(s) \colonequals \Gamma_\R(s)^{r_1} \, \Gamma_\C(s)^{r_2} \, \zeta_K(s).
\]
(Warning: 
We have used the definitions of \cite{Weil1967}*{VII.\S6}, 
but other authors use definitions of $\Gamma_\R(s)$ and $\Gamma_\C(s)$
that are nonzero constant multiples of these; 
cf.\ \cite{Deligne1973-L}*{\S3.2} and \cite{Deninger1991}*{\S2}.
Some authors also include a factor $\lvert \Disc K \rvert^{s/2}$ 
in the definition of $\widehat{\zeta}_K(s)$ \cite{Neukirch1999}*{p.~467}.)

\begin{lemma}
\label{L:completed zeta function}
We have
\[
	\lim_{s \to 1} (s-1) \widehat{\zeta}_K(s) = \frac{\vol\left( \frac{\Adeles^\times_1}{K^\times} \right)}{\lvert \Disc K \rvert^{1/2} \, [\log q]}.
\]
\end{lemma}

\begin{proof}
See the proofs of \cite{Weil1967}*{Theorems 3 and~4}.
\end{proof}

\begin{theorem}[Analytic class number formula for $S$-integers]
\label{T:ACNF for S-integers}
We have
\[
	\lim_{s \to 1} (s-1) \zeta_{\sO}(s) 
	= \frac{ 2^{r_1} \, (2\pi)^{r_2} \, h(\sO) \, R(\sO) \,
		\prod_{v \in S_{\nonarch}} 
			\left( \left( 1-q_v^{-1} \right)/\log q_v \right)}
		{w(\sO) \, \lvert \Disc K \rvert^{1/2} }.
\]
\end{theorem}

\makeatletter
\g@addto@macro\bfseries{\boldmath}
\makeatother

\begin{proof}
The functions $\widehat{\zeta}_K(s)$ and $\zeta_{\sO}(s)$
differ only in that the former contains
\begin{itemize}
\item
gamma factors $\Gamma_\R(s)$ and $\Gamma_\C(s)$,
which take the value $1$ at $s=1$,
and 
\item
Euler factors $(1-q_v^{-s})^{-1}$ for each $v \in S_{\nonarch}$,
which take the value $(1-q_v^{-1})^{-1}$ at $s=1$.
\end{itemize}
The formula follows from this, the last formula of Lemma \ref{L:volume calculations}, and Lemma~\ref{L:completed zeta function}.
\end{proof}

\section{Characterizing rings of \texorpdfstring{$S$}{S}-integers}
\label{S:characterizing}

\begin{lemma}
\label{L:characterization of S-integers}
Let $\sO$ be an integrally closed domain 
that is finitely generated as a $\Z$-algebra.
Let $K=\Frac \sO$.
If the Krull dimension $\dim \sO$ is $1$,
then $K$ is a global field and $\sO$ is a ring of $S$-integers in $K$
in the sense of Section~\textup{\ref{S:Tate}}.
\end{lemma}

\begin{proof}
\emph{Case 1: The image of $\Spec \sO \to \Spec \Z$ is a closed point.}
Then $\sO$ is a $1$-dimensional algebra over $\F_p$,
so $\Spec \sO$ is a regular curve, equal to $C-S$,
where $C$ is the smooth projective curve with function field $K$,
and $S$ is a nonempty finite set of places.

\emph{Case 2: The morphism $\Spec \sO \to \Spec \Z$ is dominant.}
Then it is of relative dimension $0$ since $\dim \sO = 1 = \dim \Z$.
Thus $K$ is a finite extension of $\Q$.
Since $\sO$ is integrally closed, 
$\sO$ contains the integral closure $\sO_K$ of $\Z$ in $K$.
Thus $\sO = \sO_K[S^{-1}]$ for some set $S$ of places of $K$.
Since $\sO$ is finitely generated, $S$ is finite.
\end{proof}

\section{Describing \texorpdfstring{$1$}{1}-dimensional schemes}
\label{S:describing}

{}From now on, 
$X$ is a reduced affine finite-type $\Z$-scheme of pure dimension~$1$,
say $X = \Spec \sO$.
Let $(C_i)_{i \in I}$ be the $1$-dimensional irreducible components of $X$.
Let $K_i$ be the function field of $C_i$.
Let $K$ be the total quotient ring of $\sO$, 
obtained by inverting all non-zerodivisors,
so $K = \prod K_i$.

Let $|X|$ be the set of closed points of $X$.
These points correspond to maximal ideals of $\sO$,
which are exactly the prime ideals with finite residue field
since $\sO$ is 
a finitely generated $\Z$-algebra \cite{EGA-IV.III}*{10.4.11.1(i)}.

Let $\pi \colon \Xtilde \to X$ 
and $\pi_i \colon \Ctilde_i \to C_i$ be the normalization morphisms.
Then $\Xtilde$ is the \emph{disjoint} union $\coprod \Ctilde_i$.
Correspondingly, the integral closure $\Otilde$ of $\sO$ in $K$
is a finite product of rings $\Otilde_i$.
By Lemma~\ref{L:characterization of S-integers}, 
$K_i$ is a global field
and there exists $S_i$ such that 
$\Otilde_i$ is the ring of $S_i$-integers in $K_i$.
Since $\pi$ is an isomorphism above the generic point of each $C_i$, 
it is an isomorphism above $X-Z$ for some finite subset $Z \subset |X|$.

\begin{lemma}
\label{L:order is of finite index}
The index $(\Otilde:\sO)$ is finite.
\end{lemma}

\begin{proof}
Since $\sO$ is a finitely generated $\Z$-algebra,
$\Otilde$ is a finite $\sO$-module.
The finite $\sO$-module $\Otilde/\sO$ is supported on $Z$,
and each $\fp \in Z$ has finite residue field.
Thus $\Otilde/\sO$ has a filtration with quotients
that are finite as sets, so $\Otilde/\sO$ is finite as a set.
\end{proof}

Our work so far proves the following.

\begin{proposition}
\label{P:characterization of rings}
A scheme $X$ is a reduced affine finite-type $\Z$-scheme of pure dimension~$1$
if and only if 
$X=\Spec \sO$ for some finite-index subring $\sO$ 
of a finite product of rings of $S_i$-integers in global fields $K_i$.
\end{proposition}

\section{Invariants}
\label{S:invariants}

We retain the notation of Section~\ref{S:describing}.

\subsection{The invariants \texorpdfstring{$m$}{m}, \texorpdfstring{$r_1$}{r1}, \texorpdfstring{$r_2$}{r2} of \texorpdfstring{$K$}{K}}

Let $m=m(K)$ be the number of irreducible components of $X$,
so $m = \# I$.

Define $r_1=r_1(K)$ and $r_2=r_2(K)$
so that $r_1$ is the number of ring homomorphisms $K \rightarrow \RR$,
and $2r_2$ is the number of ring homomorphisms $K \rightarrow \CC$
whose image is not contained in $\RR$.
If $K$ is a number field, these are the usual $r_1$ and $r_2$.
If $K$ is a global function field, then $r_1=r_2=0$.
In the general case $K = \prod K_i$,
we have $r_1(K) = \sum r_1(K_i)$ and $r_2(K) = \sum r_2(K_i)$.

\subsection{The unit group \texorpdfstring{$\sO^\times$}{Ox} and roots of unity}

Let $\sO^\times$ be the unit group of $\sO$.
Later we will prove that $\sO^\times$ is 
a finitely generated abelian group.
Let $\mu(\sO)$ be the torsion subgroup of $\sO^\times$.
Let $w(\sO):=\#\mu(\sO)$.

\subsection{The Picard group \texorpdfstring{$\Pic \sO$}{Pic O} and the class number \texorpdfstring{$h(\sO)$}{h(O)}}

Let $\Pic \sO := \Pic X = H^1(X,\OO_X^{\ast})$ \cite{Hartshorne1977}*{Exercise~III.4.5}.
Later we will prove that $\Pic \sO$ is finite.
Let $h(\sO) :=  \# \Pic \sO$.

\subsection{The discriminant \texorpdfstring{$\Disc \sO$}{Disc O}}

For each global field $K_i$, we defined $\Disc K_i$ in Section~\ref{S:Tate}.
Define $\Disc \sO \colonequals (\Otilde:\sO)^2 \, \prod \Disc K_i$;
this is so that in the case where $\sO$ is an order in the ring of integers 
of a number field, 
$\Disc \sO = \det \; (\Tr_{K/\QQ}(e_i e_j))_{1 \le i,j \le n}$
for any $\Z$-basis $(e_i)$ of $\sO$.

\subsection{The logarithmic embedding and the regulator \texorpdfstring{$R(\sO)$}{R(O)}}
\label{S:logarithmic embedding}

Let $S = \coprod S_i$ and $S_{\nonarch} = \coprod (S_i)_{\nonarch}$.
For $v \in S_i \subseteq S$, let $K_v$ be the completion of $K_i$ at $v$.
Taking the product of the homomorphisms 
$K_v^\times \stackrel{\log |\;|_v}\To \RR$
yields a homomorphism $\prod_{v \in S} K_v^\times \stackrel{\lambda}\To \RR^S$.
Let $\widetilde{\phi}$ be the composition
\[
      \Otilde^\times \To \prod_{v \in S} K_v^\times \stackrel{\lambda}\To \RR^S.
\]
Let $\phi = \widetilde{\phi}|_{\sO^\times}$.
Since $\ker \lambda$ is bounded in $\prod_{v\in S} K_v$
while $\sO^\times$ is a discrete closed subset of $\prod_{v\in S} K_v$,
$\ker \phi$ is finite;
on the other hand, the codomain of $\phi$ is torsion-free;
thus $\ker \phi = \mu(\sO)$.

The group $\widetilde{\phi}(\Otilde^\times)$ 
is a direct product of lattices in $\prod \RR^{S_i}_0$.
Later we will prove that $\sO^\times$ is of finite index in $\Otilde^\times$,
so $\phi(\sO^\times)$ is again 
a full lattice $L(\sO)$ in $\prod \RR^{S_i}_0$.
The covolume of $L(\sO)$ is called the \defi{regulator}, $R(\sO)$.

\subsection{The zeta function}

Since $\Spec \sO$ is of finite type over $\Z$,
it has a zeta function defined as an Euler product, 
as in~\cite{Serre1965}*{p.~83}:
\[
      \zeta_X(s) \colonequals \prod_{\fp \in |X|} \left(1-q_\fp^{-s}\right)^{-1}.
\]
The product converges only for $s \in \C$ with sufficiently large real part,
but as is well known and as we will explain, $\zeta_{\sO}(s)$ admits
a meromorphic continuation to the whole complex plane
and has a pole at $s=1$ of order $m$.

We have now defined all the quantities appearing in Theorem~\ref{T:main}.

\section{Relating the invariants of \texorpdfstring{$\sO$ and $\Otilde$}{O and O-tilde}}
\label{S:nonmaximal}

Theorem~\ref{T:main} for a product of rings of $S$-integers
follows from Theorem~\ref{T:ACNF for S-integers}.
In particular, it holds for $\Otilde$.
To prove it for $\sO$, 
we compare the formulas for $\sO$ and $\Otilde$ term by term.

For maximal ideals $\fp \subseteq \sO$ and $\fP \subseteq \Otilde$,
we write $\fP | \fp$ 
when $\pi$ maps the closed point $\fP \in \Xtilde$ to $\fp \in X$.

\subsection{The zeta functions of \texorpdfstring{$\sO$ and $\Otilde$}{O and O-tilde}}

Hecke \cite{Hecke1917}, generalizing Riemann's work,
proved that the Dedekind zeta function of a number field
has a meromorphic continuation to the entire complex plane
and has a simple pole at $s=1$.
The analogous result for global function fields was proved 
by F.~K.~Schmidt \cite{Schmidt1931}.
These imply the analogue for a ring of $S$-integers in a global field.
Taking a product yields the corresponding result of products
of $m$ rings of $S$-integers, except that now the pole has order $m$;
this applies in particular to $\Otilde$.
Next, by definition,
\[
      \frac{\zeta_{\Otilde}(s)}{\zeta_{\sO}(s)}
      = \prod_\fp \frac{ \prod_{\fP | \fp} (1-q_\fP^{-s})^{-1}}{(1-q_\fp^{-s})^{-1}},
\]
where, for all but finitely many $\fp$, the fraction on the right is $1$;
cf.~\cite{Jenner1969}*{Theorem}.
Thus $\zeta_{\sO}(s)$ too is meromorphic with a pole of order $m$ at $1$,
and we deduce the following.

\begin{proposition}
\label{P:ratio of zeta}
We have
\[
      \lim_{s \rightarrow 1} \frac{\zeta_{\Otilde}(s)}{\zeta_{\sO}(s)} = \prod_\fp \frac{ \prod_{\fP | \fp} (1-q_\fP^{-1})^{-1}}{\left(1-q_\fp^{-1}\right)^{-1}}.
\]
\end{proposition}

\subsection{The discriminants of \texorpdfstring{$\sO$ and $\Otilde$}{O and O-tilde}}

Our definition of $\Disc \sO$ immediately implies the following.

\begin{proposition}
\label{P:discriminants}
We have
$\displaystyle \frac{\Disc \Otilde}{\Disc \sO} = (\Otilde:\sO)^{-2}$.
\end{proposition}

\subsection{Local units}

Let $\fc \subseteq \sO$ be the annihilator of the $\sO$-module $\Otilde/\sO$.
Then $\fc$ is also an $\Otilde$-ideal, called the \defi{conductor} of $\sO$.
It is the largest $\Otilde$-ideal contained in $\sO$.

Let $\fp$ be a maximal ideal of $\sO$.
Let $\Otilde_\fp$ be the localization of the $\sO$-algebra $\Otilde$ at $\fp$.
Then $\Otilde_\fp$ is a semilocal ring whose maximal ideals correspond to 
the finitely many maximal ideals $\fP \subset \Otilde$ lying above $\fp$.
Since $\Otilde/\sO$ is finite, 
$\Otilde/\sO \isom \prod_\fp \Otilde_\fp/\sO_\fp$,
and $\Otilde_\fp/\sO_\fp$ is nontrivial for only finitely many $\fp$.
Let $\fc_\fp$ be the localization of $\fc$ at $\fp$.

\begin{lemma}
\label{L:passing to mod c}
The natural map
\[
      \frac{\Otilde_{\fp}^\times}{\sO_{\fp}^\times} \rightarrow \frac{\left( \Otilde_\fp / \fc_\fp \right)^\times}{\left( \sO_\fp / \fc_\fp \right)^\times}
\]
is an isomorphism.
\end{lemma}

\begin{proof}
\emph{Case 1: $\fc_\fp = \sO_\fp$.}
Then $1 \in \fc_\fp$, so $\fc_\fp=\Otilde_\fp$ too; thus both sides are trivial.

\emph{Case 2: $\fc_\fp \ne \sO_\fp$.}
Then $\fc_\fp \subseteq \fp \sO_\fp \subset \fP$
for every maximal ideal $\fP$ of $\Otilde_\fp$.
If an element $\bar{a} \in (\Otilde_\fp/\fc_\fp)^\times$
is lifted to an element $a \in \Otilde_\fp$,
then $a$ lies outside each $\fP$, so $a \in \Otilde_\fp^\times$.
Thus $\Otilde_\fp^\times \rightarrow (\Otilde_\fp/\fc_\fp)^\times$
is surjective.
Similarly, $\sO_\fp^\times \rightarrow (\sO_\fp/\fc_\fp)^\times$ is
surjective.  Both surjections have the same kernel $1+\fc_\fp$, so the
result follows.
\end{proof}

\begin{lemma}
\label{L:units of finite rings}
If $\fc_\fp \ne \sO_\fp$, then
\begin{align*}
      \# \left( \Otilde_\fp / \fc_\fp \right)^\times
      &= \# \left( \Otilde_\fp / \fc_\fp \right) \prod_{\fP | \fp} \left( 1-q_\fP^{-1} \right), \\
      \# \left( \sO_\fp / \fc_\fp \right)^\times
      &= \# \left( \sO_\fp / \fc_\fp \right) \left(1-q_\fp^{-1}\right).
\end{align*}
\end{lemma}

\begin{proof}
The maximal ideals of $\Otilde_\fp / \fc_\fp$ 
are the ideals $\fP \Otilde_\fp / \fc_\fp$ for $\fP | \fp$.
An element of $\Otilde_\fp / \fc_\fp$ is a unit if and only if it lies outside each
maximal ideal.
The probability that a random element of the finite group
$\Otilde_\fp / \fc_\fp$ lies outside $\fP \Otilde_\fp / \fc_\fp$ 
is $1 - q_\fP^{-1}$,
and these events for different $\fP$ are independent 
by the Chinese remainder theorem,
so the first equation follows.
The second equation is similar (but easier).
\end{proof}

\begin{lemma}
\label{L:product of local unit indices}
We have
\[
      \# \prod_\fp \frac{\Otilde_{\fp}^\times}{\sO_{\fp}^\times}
      = \# \frac{\Otilde}{\sO} \cdot \prod_\fp \frac{ \prod_{\fP | \fp} (1-q_\fP^{-1})}{1-q_\fp^{-1}}.
\]
\end{lemma}

\begin{proof}
By Lemmas \ref{L:passing to mod c} and~\ref{L:units of finite rings},
\[
      \# \frac{\Otilde_{\fp}^\times}{\sO_{\fp}^\times}
      = \# \frac{\Otilde_\fp}{\sO_\fp} \cdot \frac{ \prod_{\fP | \fp} (1-q_\fP^{-1})}{1-q_\fp^{-1}};
\]
this holds even if $\fc_\fp = \sO_\fp$ since both sides are $1$ in that case.
Now take the product of both sides
and use the isomorphism of finite groups
\[
      \frac{\Otilde}{\sO} \isom \prod_\fp \frac{\Otilde_\fp}{\sO_\fp}.\qedhere
\]
\end{proof}

\subsection{Using the Leray spectral sequence to relate units and Picard groups for \texorpdfstring{$\sO$ and $\Otilde$}{O and O-tilde}}

Let $\OO_X$ be the structure sheaf of $X$,
and let $\OO_X^\times$ be the sheaf of units of $\OO_X$.
Define $\OO_{\Xtilde}^\times$ similarly.

\begin{lemma}
\label{L:R^1}
The sheaf $R^1\pi_{\ast}\OO_{\Xtilde}^\times$ on $X$ is $0$.
\end{lemma}

\begin{proof}
By \cite{Hartshorne1977}*{Proposition~III.8.1},
its stalk
 $(R^1\pi_{\ast}\OO_{\Xtilde}^\times)_\fp$
 at a closed point $\fp$ of $X$ is $\varinjlim_U \Pic \pi^{-1} U$,
where $U$ ranges over open neighborhoods of $\fp$ in $X$.
Since $\pi^{-1}(\fp)$ is finite, every line bundle on $\pi^{-1} U$
becomes trivial on $\pi^{-1} U'$ 
for some smaller neighborhood $U'$ of $\fp$ in $X$.
Thus $\varinjlim_U \Pic \pi^{-1} U = 0$.
\end{proof}

\begin{lemma}
\label{L:H^1 and Pic}
We have $H^1(X,\pi_{\ast}\OO_{\Xtilde}^{\times}) \isom \Pic \Xtilde$.
\end{lemma}

\begin{proof}
The Leray spectral sequence
\[
      H^p\left(X, R^{q}\pi_{\ast}\scriptF\right)\implies H^{p+q}\left(\Xtilde, \scriptF\right)
\]
with $\scriptF = \OO_{\Xtilde}^\times$ yields an exact sequence
\[
0\To H^1\left(X,\pi_{\ast}\OO_{\Xtilde}^{\times}\right)\To
\Pic \Xtilde\To
H^{0}\left(X,R^1\pi_{\ast}\OO_{\Xtilde}^\times\right) .
\]
Lemma~\ref{L:R^1} above completes the proof.
\end{proof}

\begin{proposition}
\label{P:Neukirch}
The following is an exact sequence of finite groups:
\begin{equation}
\label{E:Neukirch}
	0
	\To \frac{\Otilde^\times}{\sO^\times} 
	\To \Directsum_{\fp} \frac{\Otilde_\fp^\times}{\sO_\fp^\times} 
	\To \Pic \sO 
	\To \Pic \Otilde 
	\To 0.
\end{equation}
\end{proposition}

\begin{proof}
View $\Otilde_{\fp}^\times / \sO_{\fp}^\times$ as a skyscraper sheaf
on $X$ supported at $\fp$.
Then we have an exact sequence of sheaves on $X$
\[
	0 
	\To \OO_X^\times
	\To \pi_{\ast}\OO_{\Xtilde}^\times
	\To \Directsum_\fp \frac{\Otilde_\fp^\times}{\sO_\fp^\times}
	\To 0.
\]
The corresponding long exact sequence in cohomology is
\[
	0
	\To \sO^\times
	\To \Otilde^\times
	\To \Directsum_{\fp} \frac{\Otilde_{\fp}^\times}{\sO_{\fp}^\times} 
	\To \Pic X 
	\To H^1\left(X, \pi_{\ast}\OO_{\Xtilde}^\times\right)
	\To 0.
\]
By definition, $\Pic X = \Pic \sO$.
By Lemma~\ref{L:H^1 and Pic}, the last term is $\Pic \Xtilde = \Pic \Otilde$.

The second term in \eqref{E:Neukirch} is finite 
by Lemma~\ref{L:product of local unit indices}.
Finally, $\Otilde$ is a finite product of rings of $S$-integers,
each of which has finite Picard group,
so $\Pic \Otilde$ is finite.
Thus all four groups in \eqref{E:Neukirch} are finite.
\end{proof}

\begin{remark}
For a more elementary derivation of~\eqref{E:Neukirch},
at least in the case where $\sO$ is an integral domain, 
see \cite{Neukirch1999}*{Proposition~I.12.9}.
\end{remark}

\begin{proposition}
\label{P:family}
We have
\[
      \# \frac{\Otilde^\times}{\sO^\times}
      = \frac{h(\Otilde)}{h(\sO)} \cdot \# \frac{\Otilde}{\sO} \cdot \prod_\fp \frac{ \prod_{\fP | \fp} (1-q_\fP^{-1})}{1-q_\fp^{-1}}.
\]
\end{proposition}

\begin{proof}
Take the alternating product of the orders of the groups in~\eqref{E:Neukirch}
and use Lemma~\ref{L:product of local unit indices}.
\end{proof}

\begin{remark}
Finiteness of $(\Otilde^\times:\sO^\times) < \infty$
can also be viewed as a consequence of the finiteness of $(\Otilde:\sO)$,
by \cite{Bartel-Lenstra2017}*{Theorem~1.3}.
\end{remark}

\subsection{The regulators of \texorpdfstring{$\sO$ and $\Otilde$}{O and O-tilde}}

Let $L = L(\sO)$ be as in Section~\ref{S:logarithmic embedding},
and define $\Ltilde = L(\Otilde)$ similarly.
The group $\Ltilde$ is a full lattice in $\prod \R^{S_i}_0$.
By Proposition~\ref{P:Neukirch}, $(\Otilde^\times:\sO^\times)$ is finite,
so $L$ is a full lattice in $\prod \R^{S_i}_0$ too.

\begin{proposition}
\label{P:regulators}
We have
\[
      \frac{R(\Otilde)}{R(\sO)} \cdot \# \frac{\Otilde^\times}{\sO^\times}
      = \frac{w(\Otilde)}{w(\sO)}.
\]
\end{proposition}

\begin{proof}
Applying the snake lemma to
\[
\xymatrix{
1 \ar[r] & \mu(\sO) \ar[r] \ar@{^{(}->}[d] & \sO^\times \ar[r] \ar@{^{(}->}[d] & L \ar[r] \ar@{^{(}->}[d] & 0 \\
1 \ar[r] & \mu(\Otilde) \ar[r] & \Otilde^\times \ar[r] & \Ltilde \ar[r] & 0 \\
}
\]
yields an exact sequence
\[
      1\To \frac{\mu(\Otilde)}{\mu(\sO)}
      \To \frac{\Otilde^\times}{\sO^\times}
        \To \frac{\Ltilde}{L}\To 0
\]
of finite groups, the last of which has order $R(\sO)/R(\Otilde)$.
\end{proof}

\subsection{Conclusion of the proof}

To complete the proof of Theorem~\ref{T:main},
we compare \eqref{E:sweet potato} for $\Otilde$ 
to \eqref{E:sweet potato} for $\sO$.
The ratio of the left side of \eqref{E:sweet potato} for $\Otilde$
to the left side of \eqref{E:sweet potato} for $\sO$ is
\[
      \lim_{s \rightarrow 1} \frac{\zeta_{\Otilde}(s)}{\zeta_{\sO}(s)}.
\]
The ratio of the right sides is
\[
      \left| \frac{\Disc \Otilde}{\Disc \sO} \right|^{-1/2}
      \left( \frac{w(\Otilde)}{w(\sO)} \right)^{-1} \cdot
      \frac{h(\Otilde)}{h(\sO)} \cdot
      \frac{R(\Otilde)}{R(\sO)}.
\]
By Propositions \ref{P:ratio of zeta}, \ref{P:discriminants}, \ref{P:family}, and~\ref{P:regulators} and the definition of $\Disc \sO$, 
both ratios equal
\[
      \prod_\fp \frac{ \prod_{\fP | \fp} (1-q_\fP^{-1})^{-1}}
			{\left(1-q_\fp^{-1}\right)^{-1}}.
\]

\section{Example 1: a fiber product of rings}
\label{S:example 1}

Let $p$ be an odd prime.
Consider the fiber product of \emph{rings}\footnote{A fiber product of rings does not correspond to a fiber product of schemes.} 
\[
      \sO := \Z[1/2]\times_{\F_p}\F_p[t]
	= \{\, (a,f) \in \Z[1/2] \times \F_p[t] : a \equiv f(0) \!\!\! \pmod{p} \,\}.
\]
Then $\Otilde = \Z[1/2] \times \F_p[t]$ inside $K=\QQ \times \F_p(t)$.
Thus $\Xtilde \colonequals \Spec \Otilde$ 
is the disjoint union of two ``curves''
$\Spec \Z[1/2] \, \amalg \, \Spec \F_p[t]$,
and $X \colonequals \Spec \sO$ 
is the same except that the points $(p) \in \Spec\Z[1/2]$ 
and $(t) \in \Spec \F_p[t]$ are attached.
Define 
\begin{align*}
      \fp &:= \{\, (a,f)\in \Z[1/2]\times\F_p[t] : a\equiv f(0)\equiv 0 \!\! \pmod{p} \,\} \\
      \fP &:= \{\, (a,f)\in \Z[1/2]\times\F_p[t] : a\equiv 0 \!\! \pmod{p} \,\} \\
      \fP' &:= \{\, (a,f)\in \Z[1/2]\times\F_p[t] : f(0)\equiv 0 \!\! \pmod{p} \,\}.
\end{align*}
Then $\fp$ is a prime of $\sO$ (the point of attachment),
and $\fP$ and $\fP'$ are the primes of $\Otilde$ lying above $\fp$.
The conductor of $\sO$ is $\fp$ viewed as an $\Otilde$-ideal.

Propositions \ref{P:example LHS1} and~\ref{P:example RHS1} below
verify Theorem~\ref{T:main} for $\sO$
by computing the two sides of \eqref{E:sweet potato} independently.

\begin{proposition}
\label{P:example LHS1}
We have
\[
      \lim_{s \to 1} (s-1)^2 \zeta_X(s) = \frac{1-p^{-1}}{2 \log p}.
\]
\end{proposition}

\begin{proof}
We have
\begin{align*}
	\lim_{s \to 1} (s-1) \zeta_{\Z}(s) &= 1, \\
	\lim_{s \to 1} (s-1) \zeta_{\F_p[t]}(s) &= 
		\lim_{s \to 1} \frac{s-1}{1-p^{1-s}} = \frac{1}{\log p}.
\end{align*}
Taking zeta functions of 
\[
	X-\{\fp\} = \Xtilde-\{\fP,\fP'\} = \left( \Spec \Z - \{(2),(p)\} \right) \, \amalg \, \left(\Spec \F_p[t] - \{(t)\} \right)
\]
yields
\begin{align*}
	(1-p^{-s}) \, \zeta_X(s) 
		&= (1-2^{-s})(1-p^{-s}) \, \zeta_\Z(s) \cdot (1-p^{-s}) \, \zeta_{\F_p[t]}(s) \\
	(s-1)^2 \, \zeta_X(s) 
		&= (1-2^{-s})(1-p^{-s}) \, \left((s-1) \, \zeta_\Z(s) \right) \left((s-1) \, \zeta_{\F_p[t]}(s)\right) \\
	\lim_{s \to 1} (s-1)^2 \zeta_X(s) 
		&= (1-2^{-1})(1-p^{-1}) \cdot 1 \cdot \frac{1}{\log p} = \frac{1-p^{-1}}{2 \log p}.\qedhere
\end{align*}

\end{proof}

\begin{proposition}
\label{P:example RHS1}
We have
\[
\frac{2^{r_1} \, (2\pi)^{r_ 2}\, h(\sO) \, R(\sO) \prod_{v \in S_{\nonarch}} \left( \left( 1-q_v^{-1} \right)/ \log q_v \right)}
		{w(\sO) \, \lvert \Disc \sO \rvert^{1/2}}
      = \frac{1-p^{-1}}{2\log p}.
\]
\end{proposition}

\begin{proof}
First, $r_1=1$ and $r_2=0$.
The set $S_{\nonarch}$ consists of the place $2$ of $\Q$ 
and the place $\infty$ of $\F_p(t)$.
By definition,
\[
	\Disc \sO 
	= (\Otilde:\sO)^2 \, (\Disc \Q) \, (\Disc \F_p(t))
	= p^2 \cdot 1 \cdot p^{2\cdot 0 -2} =1.
\]
Inside 
$\Otilde^\times = \Z[1/2]^\times \times \F_p^\times = \pm \{2^{n}\}_{n\in\Z} \times \F_p^\times$,
 we have 
\[
      \sO^\times = \{\pm (2^n,\,2^n\bmod p) : n \in \Z \}.
\]
In particular, $\mu(\sO) = \{\pm 1\}$, so $w(\sO) = 2$.
Since $\sO^\times$ and $\Otilde^\times$ agree modulo torsion,
\[
	R(\sO) = R(\Otilde) = R(\Z[1/2]) \, R(\F_p[t]) 
	= (\log 2) \cdot 1 = \log 2.
\]
By Lemma~\ref{L:passing to mod c}, 
\[
	\frac{\Otilde_\fp^\times}{\sO_\fp^\times} 
	\isom \frac{(\Otilde/\fP)^\times \times (\Otilde/\fP')^\times}
		{(\sO/\fp)^\times}
	\isom \frac{\F_p^\times \times \F_p^\times}{\F_p^\times},
\]
in which the denominator $\F_p^\times$ is embedded diagonally
in $\F_p^\times \times \F_p^\times$.
Substituting into~\eqref{E:Neukirch} yields
\[
	1
        \To \frac{\Otilde^\times}{\sO^\times}
	\To \frac{\F_p^\times \times \F_p^\times}{\F_p^\times}
        \To \Pic \sO
        \To \Pic \Otilde
	\To 1.
\]
The subgroup $1 \times \F_p^\times$ of $\Otilde^\times$ surjects onto 
$\dfrac{\F_p^\times \times \F_p^\times}{\F_p^\times}$,
and $\Pic \Otilde = \Pic \Z[1/2] \times \Pic \F_p[t] = \{1\}$,
so $\Pic \sO = \{1\}$.
Thus $h(\sO)=1$.

Substituting all these values shows that the expression to be computed equals
\[
    \frac{2^1 \, (2\pi)^0 \cdot 1 \cdot (\log 2) \cdot ((1 - 2^{-1})/ \log 2) \, ((1-p^{-1})/\log p)}
		{2 \cdot 1^{1/2}}
     = \frac{1-p^{-1}}{2\log p}.\qedhere
\]
\end{proof}

\section{Example 2: a non-maximal order in a real quadratic number field}
\label{S:example 2}

Let $\sO = \Z[\sqrt{d}]$,
where $d$ is an integer such that $d \ge 5$ and $d \equiv 5 \pmod{8}$
(other cases could be handled similarly).
Then $\Otilde = \Z[(1+\sqrt{d})/2]$.
Above the prime ideal $\fp \colonequals (2,1+\sqrt{d})$ of $\sO$ 
with residue field $\F_2$
lies the prime ideal $(2)$ of $\Otilde$ with residue field $\F_4$.
The scheme $X \colonequals \Spec \sO$ is analogous to a nodal curve
for which the two slopes at the node $\fp$ are conjugate in 
the quadratic extension $\F_4$ of $\F_2$.
Let $\tilde{\epsilon}$ be the fundamental unit of $\Otilde^\times$,
so $\tilde{\epsilon}>1$ for the standard real embedding $\Otilde \injects \R$.
Let $n$ be the order of the image of $\tilde{\epsilon}$ 
in $(\Otilde/(2))^\times \isom \F_4^\times$, so $n$ is $1$ or $3$.
Since $\sO$ is the preimage of $\F_2$ under $\Otilde \surjects \F_4$,
the element $\epsilon \colonequals \tilde{\epsilon}^n$ 
is the smallest power of $\tilde{\epsilon}$ lying in $\sO^\times$.

Propositions \ref{P:zeta in example 2} and~\ref{P:RHS in example 2} below
verify Theorem~\ref{T:main} for $\sO$
by computing the two sides of \eqref{E:sweet potato} independently.

\begin{proposition}
\label{P:zeta in example 2}
We have
\[
      \lim_{s \to 1} (s-1)^2 \zeta_{\sO}(s) = \frac{3 \, h(\Otilde) \, \log \tilde{\epsilon}}{2 \sqrt{d}}.
\]
\end{proposition}

\begin{proof}
The classical analytic class number formula for $\Otilde$, 
with $r_1=1$, $r_2=0$, $R(\Otilde) = \log \tilde{\epsilon}$,
$w(\Otilde)=2$, $\Disc \Otilde = d$, yields
\[
	\lim_{s \to 1} (s-1) \zeta_{\Otilde}(s) = \frac{h(\Otilde) \, \log \tilde{\epsilon}}{\sqrt{d}}.
\]
On the other hand, Proposition~\ref{P:ratio of zeta}
with $q_{\fp}=2$ and $q_{\fP}=4$ yields
\[
	\lim_{s \rightarrow 1} \frac{\zeta_{\Otilde}(s)}{\zeta_{\sO}(s)} = \frac{(1-1/4)^{-1}}{(1-1/2)^{-1}} = \frac{2}{3}.
\]
Dividing the first equation by the second gives the result.
\end{proof}

\begin{proposition}
\label{P:RHS in example 2}
We have
\[
\frac{2^{r_1} \, (2\pi)^{r_ 2}\, h(\sO) \, R(\sO) \prod_{v \in S_{\nonarch}} \left( \left( 1-q_v^{-1} \right)/ \log q_v \right)}
		{w(\sO) \, \lvert \Disc \sO \rvert^{1/2}}
      = \frac{3 \, h(\Otilde) \, \log \tilde{\epsilon}}{2 \sqrt{d}}.
\]
\end{proposition}

\begin{proof}
First, $r_1=1$, $r_2=0$, and $w(\sO)=2$.
By definition, $S_{\nonarch} = \emptyset$.
By Proposition~\ref{P:discriminants}, 
$\Disc \sO = (\Otilde:\sO)^2 \, \Disc \Otilde = 4d$,
and $R(\sO) = \log \epsilon = n \log \tilde{\epsilon}$.
The exact sequence~\eqref{E:Neukirch} is
\[
	1
        \To \frac{\Otilde^\times}{\sO^\times}
	\To \frac{\F_4^\times}{\F_2^\times}
        \To \Pic \sO
        \To \Pic \Otilde
	\To 1,
\]
so $h(\sO) = (3/n) \, h(\Otilde)$.
Thus the expression to be computed equals
\[
    \frac{2^1 \, (2\pi)^0 \cdot (3/n) \, h(\Otilde) \cdot n \log \tilde{\epsilon}}
		{2 \cdot (4d)^{1/2}}
     = \frac{3 \, h(\Otilde) \, \log \tilde{\epsilon}}{2 \sqrt{d}}.\qedhere
\]
\end{proof}

\section*{Acknowledgments} 
It is a pleasure to thank Tony Scholl for helpful discussions. 
We thank Carlos J. Moreno for bringing 
the article \cite{Stoehr1998} to our attention,
and John Voight for suggesting the reference \cite{Weil1982}.
Finally, we thank the referee for several excellent suggestions.

\end{document}